\documentclass{amsart}
\usepackage[utf8]{inputenc}
\usepackage{mathtools}
\usepackage{amsmath}
\usepackage{amssymb}
\usepackage{amsthm}
\usepackage{enumitem}
\usepackage{xcolor}
\usepackage{hyperref}
\hypersetup{
colorlinks=true, 
linkcolor= blue, 
citecolor=blue
}
\usepackage{graphicx}

\newcommand{\HRz}{H_R^{(0)}}

\DeclarePairedDelimiter{\abs}{\lvert}{\rvert}
\DeclarePairedDelimiter{\norm}{\lVert}{\rVert}

\DeclarePairedDelimiter{\set}{\lbrace}{\rbrace}
\DeclarePairedDelimiter{\br}{(}{)}
\DeclarePairedDelimiter{\sbr}{[}{]}

\newcommand{\sigess}{\sigma_e}
\newcommand{\sigdis}{\sigma_{d}}

\newcommand{\bs}{\backslash}
\newcommand{\supp}{\text{supp}}

\newcommand{\R}{\mathbb R}
\newcommand{\C}{\mathbb C}

\renewcommand{\d}{\,\text{\rm{d}}}
\newcommand{\sgn}{\text{sgn}}

\renewcommand{\geq}{\geqslant}
\renewcommand{\leq}{\leqslant}
\renewcommand{\epsilon}{\varepsilon}

\newtheorem{theorem}{Theorem}
\newtheorem{proposition}[theorem]{Proposition}
\newtheorem{lemma}[theorem]{Lemma}

\theoremstyle{remark}
\newtheorem{remark}{Remark}
\theoremstyle{definition}
\newtheorem{assumption}{Assumption}
\theoremstyle{definition}

\theoremstyle{definition}

\title[Bounds for dissipative barriers]{Bounds for Schr\"odinger operators on the half-line perturbed by dissipative barriers}

\author{Alexei Stepanenko}

\address{Cardiff University, School of Mathematics, Senghennydd Road, Cardiff, UK, CF24 4AG}

\date{\today}

\email{stepanenkoa@cardiff.ac.uk}

\thanks{The author would like to express his gratitude to his PhD supervisors Jonathan Ben-Artzi and Marco Marletta, for helpful discussion and guidance. 
The author's research is supported by the United Kingdom Engineering and Physical
Sciences Research Council, through its Doctoral Training Partnership with Cardiff University.
}



\begin{document}

\begin{abstract}

We consider Schr\"odinger operators of the form $H_R = - \d^2/\d x^2 + q + i \gamma \chi_{[0,R]}$ for large $R>0$, where $q \in L^1(0,\infty)$ and $\gamma > 0$. 
Bounds for the maximum magnitude of an eigenvalue and for the number of eigenvalues are proved. 
These bounds complement existing general bounds applied to this system, for sufficiently large $R$. 

\end{abstract}

\maketitle

%\tableofcontents

\section{Introduction}

There has recently been a surge of interest concerning bounds for the magnitude of eigenvalues and the number of eigenvalues of Schr\"{o}dinger operators with complex potentials.
In this paper, we consider Schr\"{o}dinger operators of the form
\begin{equation}\label{eq:HR-intro}
    H_R = - \frac{\d^2}{\d x^2} + q + i \gamma \chi_{[0,R]} \quad \text{on} \quad L^2(0,\infty)\qquad(R > 0),
\end{equation}
endowed with a Dirichlet boundary condition at 0, where $\gamma > 0$ and the \textit{background potential} $q \in L^1(0,\infty)$ (which may be complex-valued)  are regarded as fixed parameters.
Perturbations of the form $i \gamma \chi_{[0,R]}$ are referred to as \textit{dissipative barriers} and arise in spectral approximation, where they can be utilised as part of numerical schemes for the computation of eigenvalues \cite{marletta2012eigenvalues,stepanenko2020,db_aljawi2020,db_marletta2010ntd,db_naboko2014,db_strauss2014}.  
Our aim is to prove estimates for the magnitude and number of eigenvalues of $H_R$ for large $R$.

\subsection{Existing Bounds for the Magnitude and Number of Eigenvalues}\label{subsec:lit-rev}

Let us first discuss some relevant existing results concerning the eigenvalues of (non-self-adjoint) Schr\"odinger operators and apply them to operators of the form $H_R$. 

In \cite{abramov2001}, Abramov, Aslanyan and Davies investigated bounds for complex eigenvalues of Schr\"odinger operators, in particular obtaining a bound \cite[Theorem 4]{abramov2001} for Schr\"odinger operator on $L^2(\R)$ with a potential $V \in L^1(\R) \cap L^2(\R)$.  
Such magnitude bounds were later generalised to include more general potentials, higher dimensions and more general geometries \cite{cuenin2019improved,magnitude_daviesnath2002,magnitude_enblom2016,magnitude_Frank2011,magnitude_Frank2018,magnitude_franksimon2017,magnitude_guilarmou2019,magnitude_LaptevSafronov2009,magnitude_LeeSeo2018,magnitude_Safronov2010}. 
The work most relevant to this paper was undertook by Frank, Laptev and Seiringer \cite{FrankLaptevSeir2011}, where they show that any eigenvalue $\lambda$ of a Schr\"odinger operator $-\d^2/\d x^2 + V$ on $L^2(\R_+)$, endowed with a Dirichlet boundary condition at 0, satisfies
\begin{equation}\label{eq:lit-mag-bound}
 \sqrt{|\lambda|} \leq \norm{V}_{L^1}.
\end{equation}
Note that the right hand side of the bound presented in \cite{FrankLaptevSeir2011} depends on $\arg \lambda$ and is sharper than \eqref{eq:lit-mag-bound}.  
An application of this result to operators of the form $H_R$ gives an estimate $\sqrt{|\lambda_R|} = O(R)$ as $R \to \infty$ for any eigenvalue $\lambda_R$ of $H_R$.

Proving bounds for the number of eigenvalues of a Schr\"odinger operator is often regarded a more difficult problem. 
A sufficient condition for the potential $V$ to ensure that the number of eigenvalues of a Schr\"odinger operator on $L^2(\R_+)$ is finite is the \textit{Naimark condition} \cite{Naimark}:
\begin{equation}
 \exists a > 0: \int_0^\infty e^{at} |V(t)| \d t < \infty.
\end{equation}
There exist other such sufficient conditions and it is known that the number of eigenvalues may not be finite for certain potentials decaying only sub-exponentially \cite{bogli2017,pavlov1967,Pavlov1968}. 

\begin{table}[t]
\centering
\vspace{2mm}
\begin{tabular}{ |p{2.2cm}||p{4cm}|p{4cm}|}
\hline
 & Literature & Our Results \\
 \hline\hline
  \hfill\break  Magnitude Bound \hfill\break & \vfill $ \sqrt{|\lambda_R|} = O(R)$ \newline Frank, Laptev, Seiringer \newline (2011) \vspace{2mm}
   & \vspace{2mm} $\sqrt{|\lambda_R|} = O(R/\log R)$\newline Theorem \ref{thm:mag-bound} \\
 \hline
  Number of \newline Eigenvalues ($q$ Compactly Supported) &  \hfill\break $ N(H_R) = O(R^2)$ \newline Korotyaev (2020)
              &  \vspace{2mm}  $ N(H_R) = O(R^2/ \log R)$ \newline Theorem \ref{thm:compact-number}  \\
 \hline
  Number of \newline Eigenvalues (Naimark Condition) & \vfill $ N(H_R) = O(R^4)$ \newline Frank, Laptev, Safronov \newline (2016) \vspace{2mm}
              & \vspace{2mm} $ N(H_R) = O(R^3/(\log R)^2)$ \newline Theorem \ref{thm:exp-number} \\
 \hline
\end{tabular}
\vspace{2mm}
\caption{A summary of the large $R$ asymptotic estimates for the eigenvalues of $H_R$ implied by our results compared to estimates obtained by applying various results in the literature. }
\label{tab:asymptotics}
\end{table}

Quantitative bounds for the number of eigenvalues of a Schr\"odinger operator on $L^2(\R^d)$ were proved by Stepin in \cite{stepin2015complex,stepin2017} for dimensions $d = 1, 3$. 
Bounds for arbitrary odd dimensions were later proved by Frank, Laptev and Safronov in \cite{frank2016number}, which give better large $R$ estimates when applied to operators $H_R$ of the form \eqref{eq:HR-intro}.
\cite[Theorem 1.1]{frank2016number} states that the number of eigenvalues $N$ (counting algebraic multiplicity) of a Schr\"odinger operator $-\d^2 /\d x^2 + V$ on $L^2(\R_+)$ endowed with a Dirichlet boundary condition at 0 satisfies 
\begin{equation}
N \leq \frac{1}{\epsilon^2} \br*{\int_0^\infty e^{\epsilon t }|V(t)| \d t}^2. 
\end{equation}
for any $\epsilon > 0$.
With the assumption that the background potential $q$ satisfies the Naimark condition, applying this inequality to $H_R$ with $\epsilon = 1/R$ gives an estimate $N(H_R) = O(R^4)$ as $R \to \infty$ for the number of eigenvalues (counting algebraic multiplicities) $N(H_R)$ of $H_R$.

Additionally, Korotyaev has proved in \cite[Theorem 1.6]{korotyaev2020a} a bound specific to Schr\"odinger operators with compactly supported potentials: the number of eigenvalues $N$ of a Schr\"odinger operator $-\d^2 /\d x^2 + V$ on $L^2(\R_+)$ endowed with a Dirichlet boundary condition at 0, with $V \in L^1(\R_+)$ and $\supp V \subseteq [0,Q]$, satisfies
\begin{equation}
N \leq C_1 + C_2 Q \norm{V}_{L^1}
\end{equation}
where $C_1, C_2 > 0$ are some numerical constants.
With the assumption that the background potential $q$ is compactly supported, applying this inequality to $H_R$ gives an estimate $N(H_R) = O(R^2)$ as $R \to \infty$.
We mention also other estimates for numbers of eigenvalues in \cite{borichevCountingEigenvaluesSchr2019a,hulko2017,someyama2019}.

\subsection{Summary of Results}\label{subsec:summary}
Table \ref{tab:asymptotics} summarises our results for the large $R$ behaviour of the eigenvalues of $H_R$ and compares them to the application of the existing results to operators of the form $H_R$.

Let $\HRz$ denote the operator $H_R$ for the case $q \equiv 0$.
The semi-infinite strip
\begin{equation}\label{eq:Gam-gam-intro}
 \Gamma_\gamma := (0,\infty) + i (0,\gamma) \subset \C
\end{equation}
plays an important role throughout the paper and has the property that its closure $\overline{\Gamma}_\gamma$ is equal to the numerical range of the operator $\HRz$ for any $R > 0$.
An open ball in $\C$ of radius $r > 0$ about a point $z_0 \in \C$ is denoted by $B_r(z_0)$.
Note that in this paper we make no attempt to optimise numerical constants.

\begin{figure}[t]
  \includegraphics[width=\textwidth]{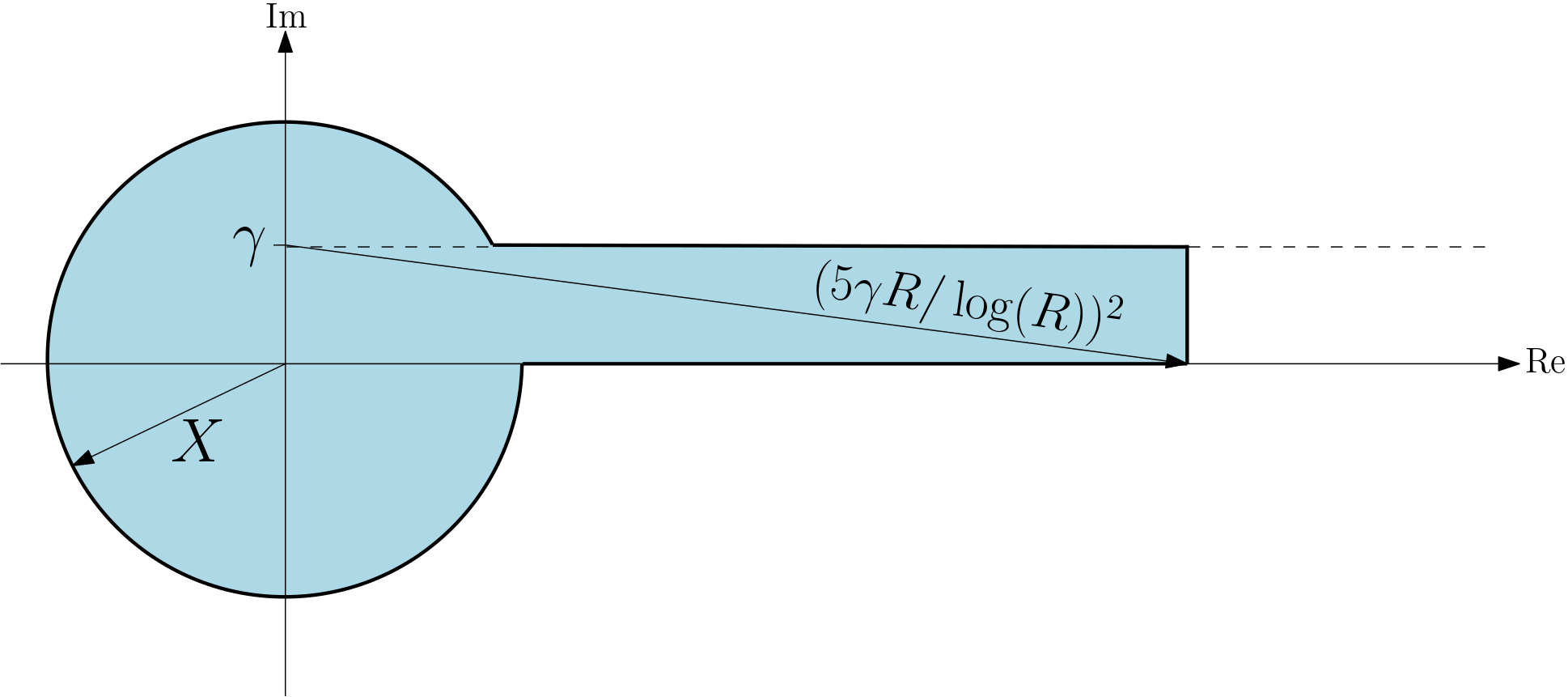}  
  \caption{Illustration the enclosure for the eigenvalues of $H_R$ provided by Theorem \ref{thm:mag-bound}.} 
  \label{fig:encl}
\end{figure}

Our first result gives a uniform in $R$ enclosure for the eigenvalues of $H_R$:
\begin{enumerate}[label=(A), labelindent=0pt]
 \item (Theorem \ref{thm:mag-bound} (a)) There exists $X = X(q,\gamma) > 0$ such that, for any $R>0$, the eigenvalues of $H_R$ lie in $B_X(0)\cup\Gamma_\gamma$.
\end{enumerate}
In particular, the imaginary and negative real components of the eigenvalues are bounded independently of $R$.

Our next result is a bound for the magnitude of eigenvalues of $H_R$ for sufficiently large $R$.
The bound gives the estimate $\sqrt{|\lambda_R|} = O(R/\log R)$ as $R \to \infty$ for any eigenvalue $\lambda_R$ of $H_R$ providing a logarithmic improvement to the application of the result \cite{FrankLaptevSeir2011} of Frank, Laptev and Seiringer to this system.
\begin{enumerate}[label=(B), labelindent=0pt]
\item (Theorem \ref{thm:mag-bound} (b)) There exists $R_0 = R_0(q, \gamma) > 0$ such that for every $R \geq R_0$, any eigenvalue $\lambda$ of $H_R$ in $\Gamma_\gamma$ satisfies
\begin{equation}\label{eq:intro-mag-bound-log}
    \sqrt{|\lambda - i \gamma|} \leq \frac{ 5 \gamma R}{\log R}.
\end{equation}
\end{enumerate}
The first inequality in (\ref{eq:lower-bounds}) shows that the estimate $\sqrt{|\lambda_R|} = O(R/\log R)$ is in fact sharp.
(B) is obtained by considering an analytic function whose zeros are the eigenvalues  of $H_R$ and applying large-$|\lambda|$ Levinson asymptotics. 
The enclosure that results from combining (A) and (B) is illustrated in Figure \ref{fig:encl}.

The fact that large eigenvalues of $H_R$ for large $R$ must be contained in the numerical range of $\HRz$ and the right hand side of inequality \eqref{eq:intro-mag-bound-log} is independent of $q$ indicates that the effect of the background potential $q$ on the large eigenvalues is dominated by effect of the dissipative barrier $i \gamma \chi_{[0,R]}$ for large $R$. 

Our first estimate for the number of eigenvalues $N(H_R)$ for $H_R$ is for the case that the background potential $q$ is compactly supported.
It gives the estimate $N(H_R) = O(R^2/ \log R)$ as $R \to \infty$, which offers a logarithmic improvement to the application of the result \cite[Theorem 1.6]{korotyaev2020a} of Korotyaev to this system.
\begin{enumerate}[label=(C), labelindent=0pt]
\item (Theorem \ref{thm:compact-number})
If $q$ is compactly supported then there exists $R_0 = R_0(q,\gamma) > 0$ such that for every $R \geq R_0$,
\begin{equation*}
    N(H_R) \leq \frac{11}{\log 2} \frac{ \gamma R^2}{\log R}.
\end{equation*}
\end{enumerate}
The second inequality in (\ref{eq:lower-bounds}) shows that the estimate  $N(H_R) = O(R^2/ \log R)$  is sharp.
The proof consists in an application of Jensen's formula.  

The case in which the background potential $q$ merely satisfies the Naimark condition requires more sophisticated techniques compared to the compactly supported case.
Our result gives the estimate $N(H_R) = O(R^3/(\log R)^2)$ as $R \to \infty$, providing a more significant improvement to the application of the result \cite[Theorem 1.1]{frank2016number} of Frank, Laptev and Safronov to this system, which gives $N(H_R) = O(R^4)$. The reasons for the more significant improvement are discussed below.
\begin{enumerate}[label=(D), labelindent=0pt]
\item (Theorem \ref{thm:exp-number})
If there exists $a > 0$ such that 
\begin{equation*}
    \int_0^\infty e^{4 at} |q(t)| \d t < \infty.
\end{equation*} 
then there exists $R_0 = R_0(q,\gamma) > 0$ such that for every $R \geq R_0$,
\begin{equation}\label{eq:intro-num-est}
    N(H_R) \leq C \frac{\sqrt{X} + a}{a^2}\frac{\gamma^2 R^3}{(\log R)^2}
\end{equation}
where $X = X(q,\gamma) > 0$ is the constant appearing in (A) and $C = 88788$.
\end{enumerate}

The proof of (D) involves first obtaining a bound which counts the number of zeros in a strip for an arbitrary analytic function in the upper half plane
(Proposition \ref{prop:Jensen-Upper}). 
This bound can be applied to the estimation of $N(H_R)$ thanks to the uniform in $R$ enclosure (A), which implies that the square-roots of the eigenvalues of $H_R$ are contained in a strip, uniformly in $R$. 
Without the uniform enclosure, we would have to use the magnitude bound (B) in place of the uniform enclosure with which the best we could obtain is inequality \eqref{eq:intro-num-est} with $\sqrt{X}$ replaced by $O(R/\log R)$, giving the large $R$ estimate $N(H_R) = O(R^4/(\log R)^3)$.
This indicates that the more significant improvement in (D) is due to the combination of a bound for the quantity $\Im \sqrt{\lambda}$ of the eigenvalues $\lambda$ with the
bound Proposition \ref{prop:Jensen-Upper} for analytic functions.

Operators of the form $H_R^{(0)}$, corresponding to the special case $q = 0$, have been studied by B\"ogli and \v{S}tampach
in \cite{bogliLiebThirringInequalities2020}, by Golinskii in \cite{golinskii2021} 
and by Cuenin in \cite{cueninSchrOdingerOperators2021}.
A consequence of \cite[Theorem 4]{cueninSchrOdingerOperators2021} is that there exists constants $C_1, C_2 > 0$ such that for all large enough $R > 0$,
\begin{equation}
  \label{eq:lower-bounds}
  \sup_{\lambda \in \sigma(H_R^{(0)})} \sqrt{|\lambda|}  \geq C_1 \frac{R}{\log(R)} \qquad \text{and} \qquad N(H_R^{(0)}) \geq C_2 \frac{R^2}{\log(R)}.  
\end{equation}
Note that although this result was formulated for the Schr\"odinger operator on $\R$, it applies to Schr\"odinger operators on $\R_+$
endowed with a Dirichlet or Neumann boundary condition since the author constructs both odd and even eigenfunctions
of $H_R^{(0)}$ in \cite[Section 7.1]{cueninSchrOdingerOperators2021}.
As already mentioned, the inequalities (\ref{eq:lower-bounds}) show that Theorem \ref{thm:mag-bound} (b) and Theorem \ref{thm:compact-number}
provide optimal large $R$ estimates. 

The reader is referred to \cite[Section 5]{stepanenko2020} for numerical illustrations of the eigenvalues of operators of the form $H_R$ for large $R$. 

\subsection{Notations and Conventions}

Throughout the paper, $C > 0$ denotes a constant, whose dependencies are generally indicated, that may change from line to line.
$\psi'(x,\lambda)$ will denote $\frac{\d}{\d x} \psi(x,\lambda)$ throughout. 
The branch cut of $\sqrt{\cdot}$ is made along $\sigess(H_R) = [0,\infty)$, so that $ \Im \sqrt{z} \geq 0 $ for all $ z \in \C$.
$N(H_R)$ shall denote the number of eigenvalues of $H_R$, counting algebraic multiplicities (as above).
Finally, note that $f_R$ will always denote an analytic function but will be redefined in each section.

\section{Magnitude Bound}\label{sec:mag-bound}

Since $q \in L^1(0,\infty)$, we can employ Levinson's asymptotic theorem which states that the solution space
of the Schr\"odinger equation $ -u'' + q u = \lambda u$ on $[0,\infty)$ is spanned by solutions $\psi_+$ and $\psi_-$,
which admit the decomposition \cite[Appendix II, Theorems 1 and 3]{naimarklinear} \cite[Theorem 1.3.1]{eastham1989asymptotic}: 
\begin{align}\label{eq:Epm-expr}
\begin{split}
    \psi_\pm(x,\lambda) & = e^{\pm i \sqrt{\lambda} x} (1 + E_\pm(x,\lambda)) \\
    \psi'_\pm(x,\lambda) & = \pm i \sqrt{\lambda} e^{\pm i \sqrt{\lambda} x} (1 + E^d_\pm(x,\lambda))
\end{split} \qquad (x \in [0,\infty),\lambda \in \C \bs \set{0}).
\end{align}
Here, $E_\pm$ and $E_\pm^d$ are some functions such that,
\begin{equation}\label{eq:Epm-limit}
|E_\pm(x,\lambda)|+|E_\pm^d(x,\lambda)| \to 0 \quad \text{as}\quad  x \to \infty 
\end{equation}
for all $\lambda \in \C \bs \set{0}$, and
\begin{equation}\label{eq:Epm-error}
    |E_\pm(x,\lambda)|+|E^d_\pm(x,\lambda)| \leq \frac{C(q)}{\sqrt{|\lambda|}}
\end{equation}
for all $x \in [0,\infty)$ and $\lambda \in \C$ with $|\lambda| \geq 1$.

While the error $E_+(x,\lambda)$ tends to 0 as $x \to \infty$ uniformly for $ \lambda \in \C \bs B_\delta(0)$, $\delta > 0$, the error $E_-$ does not have this property. 
For this reason, we will need to utilise large-$|\lambda|$ asymptotics of $\psi_\pm$ in this section.
\begin{lemma}\label{lem:fR-Lev}

$\lambda \in \C \bs [0,\infty)$ with $\lambda \neq i \gamma$ is an eigenvalue of $H_R$ if an only if $f_R(\lambda) = 0$, where
\begin{align*}
    f_R(\lambda) & := \psi_-(0,\lambda - i \gamma) \br*{ \sqrt{\lambda} - \sqrt{\lambda - i \gamma} + \mathcal{E}_1(R, \lambda)} e^{i\sqrt{\lambda - i \gamma} R} \\
    & - \psi_+(0,\lambda - i \gamma) \br*{ \sqrt{\lambda} + \sqrt{\lambda - i \gamma} + \mathcal{E}_2(R, \lambda)} e^{ - i\sqrt{\lambda - i \gamma} R}.
\end{align*}
Here, $\mathcal E_1,\mathcal E_2$ are defined, for any $R > 0$ and $\lambda \in \C \bs \set{0,i \gamma}$, by 
\begin{multline}\label{eq:big-error-E1}
    \mathcal E_1(R,\lambda)  =  \sqrt{\lambda} \br*{E_+(R,\lambda - i \gamma) + E^d_+(R,\lambda) + E_+(R,\lambda - i \gamma)E^d_+(R,\lambda) } \\
     - \sqrt{\lambda - i \gamma} \br*{ E_+^d(R,\lambda - i \gamma) + E_+(R,\lambda) + E_+^d(R,\lambda - i \gamma)E_+(R,\lambda)},
\end{multline}
\begin{multline}\label{eq:big-error-E2}
    \mathcal E_2(R,\lambda)  =  \sqrt{\lambda} \br*{ E^d_+(R,\lambda) + E_-(R,\lambda - i \gamma) + E^d_+(R,\lambda)E_-(R,\lambda - i \gamma)} \\
     + \sqrt{\lambda - i \gamma} \br*{E_+(R,\lambda) + E^d_-(R,\lambda - i\gamma) + E_+(R,\lambda)E^d_-(R,\lambda - i\gamma)}
\end{multline}
and, for some $\mathcal C_1 = \mathcal C_1(q,\gamma) > 0 $, satisfy
\begin{equation}\label{eq:big-E-bounds}
    |\mathcal E_1(R,\lambda)|+|\mathcal E_2(R,\lambda)| \leq \mathcal C_1 
\end{equation}
for all $R > 0$ and all $\lambda \in \C$ with $|\lambda| \geq 1 + \gamma$.
 Furthermore, $f_R$, $\mathcal E_1(R,\cdot)$ and $\mathcal E_2(R,\cdot)$ are analytic on $\C \bs \br*{ [0,\infty)\cup(i \gamma + [0,\infty))}$. 
\end{lemma}

\begin{proof}

Let $\lambda \in \C \bs [0,\infty)$ with $ \lambda \neq i \gamma$.
$\lambda$ is an eigenvalue of $H_R$ if and only if there a solution to the boundary value problem 
\begin{equation}\label{eq:eig-bvp}
    - \psi '' + (q + i \gamma \chi_{[0,R]})\psi = \lambda \psi \text{ on }[0,\infty),\, \psi(0) = 0, \psi \in L^2(0,\infty).
\end{equation}
Any solution to \eqref{eq:eig-bvp} on $[0,R]$ must be of the form $C_1 \psi_1(\cdot,\lambda)$, where
\begin{equation}\label{eq:psi-1}
    \psi_1(x,\lambda) := \psi_-(0,\lambda - i \gamma) \psi_+(x,\lambda - i \gamma) - \psi_+(0,\lambda - i \gamma) \psi_-(x,\lambda - i \gamma)
\end{equation}
and $C_1 \in \C$ is independent of $x$. 
Any solution to the boundary value problem \eqref{eq:eig-bvp} on $[R,\infty)$ must be of the form $C_2 \psi_+(x,\lambda)$, where $C_2 \in \C$ is independent of $x$. 
Hence $\lambda$ is an eigenvalue if and only if there exists $C_1,C_2 \in \C \bs \set{0}$ independent of $x$ such that the function
\begin{equation*}
    x \mapsto \begin{cases}
    C_1 \psi_1(x,\lambda) & \text{if }x \in [0,R) \\
    C_2 \psi_+(x,\lambda) & \text{if }x \in [R, \infty)
    \end{cases}
\end{equation*}
is continuously differentiable which holds if and only if
\begin{equation}\label{eq:trans-psi}
    i f_R(\lambda) e^{i\sqrt{\lambda} R} \equiv \psi_1(R, \lambda) \psi_+'(R, \lambda) - \psi_1'(R, \lambda) \psi_+(R, \lambda ) = 0.
\end{equation}
The required expression for $f_R$ holds by a direct computation, using expressions \eqref{eq:Epm-expr} for $\psi_\pm$.

For $\lambda \in \C$ with $|\lambda| \geq 1 + \gamma $ we have $|\lambda| \geq 1$ and $|\lambda - i \gamma| \geq 1$.
Therefore, estimates \eqref{eq:Epm-error} apply to all the terms in \eqref{eq:big-error-E1} and \eqref{eq:big-error-E2} involving $E_\pm$ or $E_\pm^d$.
The $O(1/\sqrt{|\lambda|})$ decay of the terms involving $E_\pm$ or $E_\pm^d$ as $|\lambda| \to \infty$ cancel the growth of the square roots hence estimate \eqref{eq:big-E-bounds} holds.
Finally, $f_R$, $\mathcal E_1(R,\cdot)$ and $\mathcal E_2(R,\cdot)$ are analytic on $\C \bs \br*{ [0,\infty)\cup(i \gamma + [0,\infty))}$ because $\sqrt{\cdot}$, $E_\pm(R,\cdot)$ and $E^d_\pm(R,\cdot)$ are analytic on $\C \bs [0,\infty)$.   
\end{proof}

In the special case $q \equiv 0$, $f_R$ is denoted by $f_R^{(0)}$ and we have that:
\begin{equation*}
    \lambda \in \C \bs [0,\infty)\text{ is an eigenvalue of } \HRz \text{ if and only if }f_R^{(0)}(\lambda) = 0. 
\end{equation*}
The terms $E_\pm$ and $E_\pm^d$ in Levinson's asymptotic theorem are simply zero for this case, so 
\begin{equation}\label{eq:fR0-expr}
    f_R^{(0)}(\lambda) = \br*{\sqrt{\lambda} - \sqrt{\lambda - i \gamma}} e^{ i \sqrt{\lambda - i \gamma} R} - \br*{\sqrt{\lambda} + \sqrt{\lambda - i \gamma}} e^{- i \sqrt{\lambda - i \gamma} R}.
\end{equation}

\begin{lemma}\label{lem:fR0}
There exists a constant $\mathcal C_2  = \mathcal C_2(q, \gamma) > 0$ such that 
\begin{equation*}
    |f_R(\lambda) - f_R^{(0)}(\lambda)| \leq \mathcal C_2 e^{\Im \sqrt{\lambda - i \gamma} R}
\end{equation*}
for all $R > 0$ and all $\lambda \in \C$ with $|\lambda| \geq 1 + \gamma$.
\end{lemma}

\begin{proof}
By a direct computation, using Lemma \ref{lem:fR-Lev} and the fact that 
\begin{equation*}
    \psi_\pm(0,\lambda - i \gamma) = 1 + E_\pm(0,\lambda - i \gamma),
\end{equation*}
we have 
\begin{align}\label{eqpr:fR0-1}
\begin{split}
    (f_R(\lambda) - f_R^{(0)}(\lambda))e^{i\sqrt{\lambda - i \gamma} R} & = E_-(0,\lambda - i \gamma) \sbr*{ \sqrt{\lambda} - \sqrt{\lambda - i \gamma}}e^{2 i \sqrt{\lambda - i \gamma} R} \\
    & - E_+(0,\lambda - i \gamma) \sbr*{ \sqrt{\lambda} + \sqrt{\lambda - i \gamma}} \\
    & + (1 + E_-(0,\lambda - i \gamma))\mathcal E_1(R,\lambda) e^{2i\sqrt{\lambda - i \gamma}R}  \\
    & - (1 + E_+(0,\lambda - i \gamma)) \mathcal E_2(R,\lambda).
\end{split}
\end{align}
Each term on the right hand side of \eqref{eqpr:fR0-1} is bounded uniformly for all $R > 0$ and all $\lambda \in \C$ with $|\lambda| \geq 1 + \gamma$; this follows using the boundedness for $\mathcal E_1$ and $\mathcal E_2$ proved in Lemma \ref{lem:fR-Lev} as well as the large-$|\lambda|$ asymptotics of $E_\pm(0,\lambda - i \gamma)$ in \eqref{eq:Epm-error}.
In particular, inequality \eqref{eq:Epm-error} implies that $E_\pm(0,\lambda - i \gamma) = O(1/\sqrt{|\lambda|})$ as $|\lambda| \to \infty$, balancing the growth of the factors $\sqrt{\lambda} \pm \sqrt{\lambda - i \gamma}$ in the first two terms of \eqref{eqpr:fR0-1}.
\end{proof}

Recall that $\Gamma_\gamma$ is an open strip defined by equation \eqref{eq:Gam-gam-intro}.
We shall need the following elementary inequalities:
\begin{lemma}
  \label{lem:h-pm}
\begin{enumerate}[label=\rm{(\alph*)},wide, labelindent=0pt]
    \item If $\lambda \in \Gamma_\gamma \cup [0,\infty)$ then 
    \begin{equation*}
        |\sqrt{\lambda} + \sqrt{\lambda - i \gamma}| \leq \frac{\gamma}{\sqrt{|\lambda - i \gamma|}}\text{ and }|\sqrt{\lambda} - \sqrt{\lambda - i \gamma}| \geq \sqrt{|\lambda - i \gamma|}.
    \end{equation*}
    \item If $\lambda \in \C \bs \br*{ \Gamma_\gamma \cup [0,\infty) }$ then 
    \begin{equation*}
        |\sqrt{\lambda} + \sqrt{\lambda - i \gamma}| \geq \sqrt{|\lambda|}\text{ and }|\sqrt{\lambda} - \sqrt{\lambda - i \gamma}| \leq \frac{\gamma}{\sqrt{|\lambda|}}. 
    \end{equation*}
     \item 
    If $\lambda \in \Gamma_\gamma$ then 
    \begin{equation*}
        \Im \sqrt{\lambda - i \gamma} \leq \frac{1}{\sqrt{2}} \frac{\gamma}{\sqrt{|\lambda - i \gamma|}}.
    \end{equation*}
\end{enumerate}
\end{lemma}

\begin{proof}

\begin{enumerate}[label=\rm{(\alph*)},wide, labelindent=0pt]
\item 
If $\lambda \in \Gamma_\gamma \cup [0,\infty)$ then
\begin{equation*}
    \sgn \Re \sqrt{\lambda - i \gamma} = - \sgn \Re \sqrt{\lambda},\, |\Re \sqrt{\lambda - i \gamma}| \geq \Im \sqrt{\lambda - i \gamma}\text{ and }|\Re \sqrt{\lambda}| \geq \Im \sqrt{\lambda}
\end{equation*}
so 
\begin{align}\label{eqpr:h-pm-1}
\begin{split}
    |\sqrt{\lambda} - \sqrt{\lambda - i \gamma}|^2 = & (\Re \sqrt{\lambda})^2 + (\Re \sqrt{\lambda - i \gamma})^2 + (\Im \sqrt{\lambda})^2 + (\Im \sqrt{\lambda - i \gamma})^2 \\
     & - 2 \Re \sqrt{\lambda} \Re \sqrt{\lambda - i \gamma} - 2 \Im \sqrt{\lambda} \Im \sqrt{\lambda - i \gamma} \\
     \geq & |\lambda - i \gamma|.
\end{split}
\end{align}
The inequality for $\sqrt{\lambda} + \sqrt{\lambda - i \gamma}$ follows from the identity
\begin{equation}\label{eqpr:h-pm-2}
    \sqrt{\lambda} + \sqrt{\lambda - i \gamma} = \frac{i \gamma}{ \sqrt{\lambda} - \sqrt{\lambda - i \gamma}}. 
\end{equation}
\item
If $\lambda \in i \gamma + \C_+\cup[0,\infty)$ or $\lambda \in \C_- $ then, similarly to \eqref{eqpr:h-pm-1},
\begin{equation*}
    \sgn \Re \sqrt{\lambda} = \sgn \Re \sqrt{\lambda - i \gamma} \quad \Rightarrow \quad  |\sqrt{\lambda} + \sqrt{\lambda - i \gamma}|^2 \geq |\lambda|
.\end{equation*}
If $\lambda \in (-\infty,0]+ i[0,\gamma]$ then $|\Re \sqrt{\lambda} |\leq \Im \sqrt{\lambda}$ and $|\Re \sqrt{\lambda - i \gamma} |\leq \Im \sqrt{\lambda - i \gamma}$ so  
\begin{equation*}
    |\sqrt{\lambda} + \sqrt{\lambda - i \gamma}|^2 \geq |\lambda - i \gamma| + |\lambda| \geq |\lambda|
\end{equation*}
hence the inequality for $\sqrt{\lambda} + \sqrt{\lambda - i \gamma}$ holds. The inequality for $\sqrt{\lambda} - \sqrt{\lambda - i \gamma}$ follows from \eqref{eqpr:h-pm-2}. 
\item 
Let $\lambda \in \Gamma_\gamma$ and let $z = \lambda - i \gamma$. Then $|\Im z| \leq \gamma$ so 
\begin{equation*}
    2(\Im \sqrt{z})^2 = |z| - \Re z = \frac{(\Im z)^2}{|z| + \Re z} \leq \frac{\gamma^2}{|z|}. 
\end{equation*}
\end{enumerate}
\end{proof}

Using the function $f_R$ for the eigenvalues of $H_R$, combined with the large-$|\lambda|$ asymptotics of $\psi_\pm$, we can estimate the location of the eigenvalues of $H_R$:

\begin{theorem}\label{thm:mag-bound}
\begin{enumerate}[label=\rm{(\alph*)},wide, labelindent=0pt]
\item There exists $X = X(q,\gamma) > 0$ such that, for any $R>0$, the eigenvalues of $H_R$ lie in $B_X(0)\cup\Gamma_\gamma$.
\item There exists $R_0 = R_0(q, \gamma) > 0$ such that for every $R \geq R_0$, any eigenvalue $\lambda$ of $H_R$ in $\Gamma_\gamma$ satisfies
\begin{equation}\label{eq:mag-bound-q-indpt}
    \sqrt{|\lambda - i \gamma|} \leq \frac{ 5 \gamma R}{\log R}.
\end{equation}
\end{enumerate}
\end{theorem}

\begin{proof}

\begin{enumerate}[label=\rm{(\alph*)},wide, labelindent=0pt]
\item 
Let $R>0$. 
$H_R$ has no eigenvalues in $[0,\infty)$ (indeed, this follows from the Levinson asymptotic formulas (\ref{eq:Epm-expr})) so it suffices to show that any zero of $f_R$ in $\C \bs \br*{ \Gamma_\gamma \cup [0,\infty)}$ must lie in an open ball in the complex plane, whose radius is independent of $R$.
Let $\lambda \in \C \bs \br*{ \Gamma_\gamma \cup [0,\infty)}$ be such that $|\lambda| \geq X$, where $X = X(q,\gamma) > 0$ is a large enough constant to be further specified.
Let $X>0$ be large enough so that $|\lambda| \geq 1 + \gamma$.
By the expression for $f_R$ in Lemma \ref{lem:fR-Lev}, 
\begin{align}\label{eqpr:mag-bound-a-1}
    \abs*{f_R(\lambda)e^{i\sqrt{\lambda - i \gamma}R}} 
     \geq & \left| \abs{\psi_+(0,\lambda - i \gamma)\br{\sqrt{\lambda} + \sqrt{\lambda - i \gamma} + \mathcal E_2(R,\lambda)} } \right. \nonumber \\
    & \left. - \abs{\psi_-(0,\lambda - i \gamma) \br{\sqrt{\lambda} - \sqrt{\lambda - i \gamma} + \mathcal E_1(R,\lambda)}}e^{-2\Im\sqrt{\lambda - i \gamma}R} \right|. 
\end{align}
By the boundedness of $\mathcal E_1$ and $E_-$ (Lemma \ref{lem:fR-Lev} and estimates \eqref{eq:Epm-error}), as well an inequality in Lemma \ref{lem:h-pm} (b), there exists $C_1 = C_1(q,\gamma) > 0$ such that 
\begin{equation}\label{eqpr:mag-bound-a-2}
    \abs*{\psi_-(0,\lambda -i \gamma)\br*{\sqrt{\lambda} - \sqrt{\lambda - i \gamma} + \mathcal E_1(R,\lambda)}}e^{ - 2\Im \sqrt{\lambda -i \gamma} R} \leq C_1.
\end{equation}
Let $\delta > 0$.
Recall that $|\mathcal E_2(R,\lambda)| \leq \mathcal{C}_1$, where $\mathcal C_1 > 0$ is the constant appearing in Lemma \ref{lem:fR-Lev}. 
Let $X>0$ be large enough such that $|\psi_+(0,\lambda - i \gamma)| \geq \frac{1}{2}$
and
\begin{equation*}
    \sqrt{|\lambda|} \geq 2(C_1 + \delta) + \mathcal C_1.
\end{equation*}
Then, using Lemma \ref{lem:h-pm} (b), 
\begin{equation}\label{eqpr:mag-bound-a-3}
    \abs*{\psi_+(0,\lambda - i \gamma) \br*{\sqrt{\lambda} + \sqrt{\lambda -i \gamma} + \mathcal E_2(R, \lambda)}}
     \geq \frac{1}{2} \abs*{ \sqrt{|\lambda|} - \mathcal C_1} \geq C_1 + \delta.
\end{equation}
Combining \eqref{eqpr:mag-bound-a-1}, \eqref{eqpr:mag-bound-a-2} and \eqref{eqpr:mag-bound-a-3}, we have
\begin{equation*}
    |f_R(\lambda)|\geq \delta > 0.
\end{equation*}
Consequently, $\lambda$ is not an eigenvalue of $H_R$ proving that there are no eigenvalues of $H_R$  in $\C \bs \Gamma_\gamma$ with magnitude greater than $X$.

\item

Let $R \geq R_0$, where $R_0 = R_0(q,\gamma) > 0$ is a large enough constant to be further specified. Let $\lambda \in \Gamma_\gamma$ be such that 
\begin{equation}\label{eqpr:mag-bound-b-1}
    \sqrt{|\lambda - i \gamma|} \log |\lambda - i \gamma| \geq 8 \gamma R.
\end{equation}
We aim to prove that $\lambda$ is not an eigenvalue of $H_R$.

Using the expression \eqref{eq:fR0-expr} for $f_R^{(0)}$, 
\begin{equation}
    \frac{|f_R^{(0)}(\lambda)|}{|\lambda - i \gamma|^{1/4}}e^{- \Im \sqrt{\lambda - i \gamma}R}  \geq \abs*{ \frac{|\sqrt{\lambda} - \sqrt{\lambda - i \gamma}|}{|\lambda - i \gamma|^{1/4}}e^{-2 \Im \sqrt{\lambda - i \gamma}R} -  \frac{|\sqrt{\lambda} + \sqrt{\lambda - i \gamma}|}{|\lambda - i \gamma|^{1/4}} } \nonumber
\end{equation}
Using the inequality \eqref{eqpr:mag-bound-b-1} and Lemma \ref{lem:h-pm} (c), $\lambda$ satisfies
\begin{equation}\label{eqpr:mag-bound-b-2}
    e^{-2\Im\sqrt{\lambda - i \gamma}R} \geq e^{- \sqrt{2} \gamma R/\sqrt{|\lambda - i \gamma|}} \geq e^{-\frac{\sqrt{2}}{8}\log |\lambda - i \gamma|} = \frac{1}{|\lambda - i \gamma|^{\sqrt{2}/8}}.
\end{equation}
Ensure $R_0 > 0$ is large enough so that $|\lambda - i \gamma|^{1/4} \geq 2 |\lambda - i \gamma|^{\sqrt{2}/8}$. Then, using Lemma \ref{lem:h-pm} (a), 
\begin{equation}\label{eqpr:mag-bound-b-3}
    \frac{|\sqrt{\lambda} - \sqrt{\lambda - i \gamma}|}{|\lambda - i \gamma|^{1/4}} \geq |\lambda - i \gamma|^{1/4} \geq 2 |\lambda - i \gamma|^{\sqrt{2}/8}. 
\end{equation}
Ensure also that $R_0 > 0$ is large enough so that $|\lambda - i \gamma| \geq \gamma^{4/3}$.
Combining \eqref{eqpr:mag-bound-b-3} with \eqref{eqpr:mag-bound-b-2} and using Lemma \ref{lem:h-pm} (a) again,
\begin{equation*}
    \frac{|\sqrt{\lambda} - \sqrt{\lambda - i \gamma}|}{|\lambda - i \gamma|^{1/4}}e^{-2 \Im \sqrt{\lambda - i \gamma}R} \geq 2 \geq 1 + \frac{\abs*{\sqrt{\lambda} + \sqrt{\lambda - i \gamma}}}{|\lambda - i \gamma|^{1/4}}.
\end{equation*}
and hence
\begin{equation}\label{eqpr:mag-bound-2}
    |f_R^{(0)}(\lambda)| \geq |\lambda - i \gamma|^{1/4} e^{\Im \sqrt{\lambda - i \gamma}R}.
\end{equation}
In particular, $f_R^{(0)}(\lambda) \neq 0$. 

Recall that $\mathcal C_2 = \mathcal C_2(q,\gamma) > 0$ denotes the constant appearing in Lemma \ref{lem:fR0}. 
Ensure that $R_0 > 0$ is large enough so that  $|\lambda| \geq 1 + \gamma$ and $|\lambda - i \gamma|^{1/4} \geq 2 \mathcal{C}_2$. 
By \eqref{eqpr:mag-bound-2} and Lemma \ref{lem:fR0},
\begin{equation*}
    |f_R(\lambda) - f_R^{(0)}(\lambda)| \leq \mathcal C_2 e^{\Im \sqrt{\lambda - i \gamma}R} \leq \frac{1}{2} |\lambda - i \gamma|^{1/4} e^{\Im \sqrt{\lambda - i \gamma}R} \leq \frac{1}{2}  |f_R^{(0)}(\lambda)|
\end{equation*}
therefore $f_R(\lambda)\neq 0$ and, consequently, $\lambda$ is not an eigenvalue of $H_R$. 
This proves that any eigenvalue of $H_R$ must satisfy 
\begin{equation}\label{eqpr:mag-bound-3}
    \sqrt{|\lambda - i \gamma|} \log \sqrt{|\lambda - i \gamma|} \leq 4 \gamma R.
\end{equation}

Let $W$ denote the Lambert-$W$-function (also known as the product log function). 
$W$ satisfies 
\begin{equation*}
 W(x) = \log \br*{\frac{x}{W(x)}} \quad \text{and} \quad  y \log y = x \iff y = \frac{x}{W(x)}   \qquad (x > 0, y > 0).
\end{equation*}
Hence \eqref{eqpr:mag-bound-3} can be written as
\begin{equation*}
    \sqrt{|\lambda - i \gamma|} \leq \frac{4 \gamma R}{W(4 \gamma R)} =\frac{4 \gamma R}{\log(4\gamma R) -  \log(W(4 \gamma R))}
\end{equation*}
from which \eqref{eq:mag-bound-q-indpt} follows. 
\end{enumerate}
\end{proof}

\begin{remark}
The constant $X = X(q,\gamma) > 0$ in Theorem \ref{thm:mag-bound} (a) satisfies 
\begin{equation*}
    X = O(\norm{q}_{L^1}^3) \quad \text{as} \quad \norm{q}_{L^1} \to \infty.
\end{equation*}
This can be seen by noting that $E_\pm(R,\lambda),E_\pm^d(R,\lambda) = O(\norm{q}_{L^1})$ (see \cite[Chapter 1.4]{eastham1989asymptotic}), $\mathcal C_1 = O(\norm{q}^2_{L^1})$ and $C_1 = O(\norm{q}^3_{L^1})$.
\end{remark}

\section{Number of Eigenvalues}\label{sec:number}

In this section, we estimate the number of eigenvalues for $H_R$, for which we necessarily need to add additional assumptions on the background potential $q$. 

\subsection{Preliminaries}

Let $\psi_\pm$ denote the solutions \eqref{eq:Epm-expr} for the Schr\"odinger equation and 
\begin{equation*}
    \varphi(x,z) := \psi_+(x,z^2) \qquad (x \in [0,\infty),z \in \C_+).
\end{equation*}
$\varphi$ is commonly referred to as the \textit{Jost solution}. 
For each $R > 0$, define function $f_R: \C_+ \to \C$ by 
\begin{equation*}
    i f_R(z)e^{i z R} =  \theta(R,z)\varphi'(R,z) - \theta'(R,z) \varphi(R,z) \qquad (z \in \C_+).
\end{equation*}
where, for any $z \in \C$, $\theta(\cdot,z)$ is defined as the solution to the initial value problem 
\begin{equation}\label{eq:IVP-dirichlet}
    - \theta'' + q \theta  = (z^2 - i \gamma) \theta,\,\theta(0) = 0, \theta'(0) = 1.
\end{equation}
By the same arguments as in Lemma \ref{lem:fR-Lev}, we have the following. 
\begin{lemma}
    $f_R$ is analytic on $\C_+$ and any $z \in \C_+$ satisfies 
    \begin{equation}
     f_R(z) = 0 \quad \iff \quad z^2 \text{ is an eigenvalue of }H_R.
    \end{equation}
\end{lemma}
$\varphi$ can be decomposed in a similar way to $\psi_\pm$, 
\begin{align}\label{eq:phi-decomp}
\begin{split}
    \varphi(x,z) & = e^{ i z x} (1 + E(x,z)) \\
    \varphi'(x,z) & =  i z e^{i z x} (1 + E^d(x,z))
\end{split}\qquad (x \in [0,\infty),z \in \C_+)
\end{align}
for some functions $E$ and $E^d$ whose properties will be later specified, for the different assumptions on the background potential $q$ that we consider.
We shall need the following facts concerning $f_R$ and $\theta$.
Note that in Lemma \ref{lem:cont-fR-decomp}, $\mathcal E_1$ and $\mathcal E_2$ are defined in a   different way than in Lemma \ref{lem:fR-Lev}.
\begin{lemma}\label{lem:cont-fR-decomp}
Suppose that, for each $R > 0$, $\varphi(R,\cdot)$ and $\varphi'(R,\cdot)$ admits an analytic continuation from $\C_+$ into some open $U \subset \C$.  
Then, $f_R$ admits analytic continuation into $U$. Furthermore, for each $R > 0$ and $z \in U \bs \set{\pm \sqrt{i \gamma}}$,
\begin{multline}
    f_R(z)u(z)  = \psi_-(0, z^2 - i \gamma) \br*{ z - \sqrt{z^2 - i \gamma} + \mathcal E_1(R,z) }e^{i \sqrt{z^2 - i \gamma} R} \\
     - \psi_+(0, z^2 - i \gamma) \br*{ z + \sqrt{z^2 - i \gamma} + \mathcal E_2(R,z)}e^{- i \sqrt{z^2 - i \gamma} R}
\end{multline}
where
\begin{equation}\label{eq:u-defn}
 u(z) := \psi_-(0,z^2 - i \gamma)\psi'_+(0,z^2 - i \gamma) - \psi_+(0,z^2 - i \gamma)\psi'_-(0,z^2 - i \gamma),
\end{equation}
\begin{multline}\label{eq:big-error-cont-E1}
    \mathcal E_1(R,z)  :=  z \br*{E_+(R,z^2 - i \gamma) + E^d(R,z) + E_+(R,z^2 - i \gamma)E^d(R,z) } \\
     - \sqrt{z^2 - i \gamma} \br*{ E_+^d(R,z^2 - i \gamma) + E(R,z) + E_+^d(R,z^2 - i \gamma)E(R,z)}
\end{multline}
and
\begin{multline}\label{eq:big-error-cont-E2}
    \mathcal E_2(R,z)  :=  z \br*{ E^d(R,z) + E_-(R,z^2 - i \gamma) + E^d(R,z)E_-(R,z^2 - i \gamma)} \\
     + \sqrt{z^2 - i \gamma} \br*{E(R,z) + E^d_-(R,z^2 - i\gamma) + E(R,z)E^d_-(R,z^2 - i\gamma)}.
\end{multline}
\end{lemma}
\begin{proof}
Analytic continuation holds by the fact that $\theta(R,\cdot)$ is entire \cite[Lemma 5.7]{teschlOrdinaryDifferentialEquations2012} for each $R > 0$.
If $z \neq \pm \sqrt{i \gamma}$ then the functions $\psi_\pm(\cdot,z^2 -i \gamma)$ span the solution space of the Schr\"odinger equation $ - \psi'' + q \psi = (z^2 - i \gamma) \psi$ so 
\begin{align*}
    \theta(R,z) &  = \frac{\psi_-(0,z^2 - i \gamma) \psi_+(R,z^2 - i \gamma) - \psi_+(0,z^2 - i \gamma) \psi_-(R,z^2 - i \gamma)}{\psi_-(0,z^2 - i \gamma)\psi'_+(0,z^2 - i \gamma) - \psi_+(0,z^2 - i \gamma)\psi'_-(0,z^2 - i \gamma)} \\
    & = \frac{\psi_1(R,z^2)}{u(z)}
\end{align*}
where $\psi_1$ denotes the function defined by \eqref{eq:psi-1} in Lemma \ref{lem:fR-Lev}.
The lemma follows by a direct computation, similar to one in Lemma \ref{lem:fR-Lev}.
\end{proof}

\begin{lemma}\label{lem:gronwall-est}
   For any $x \in [0,\infty)$ and $z \in \C \bs \set{\pm i \gamma}$, the solution $\theta$ to the initial value problem (\ref{eq:IVP-dirichlet}) satisfies the inequality
    \begin{equation*}
        |\theta(x,z)|+|\theta'(x,z)| \leq (1 + x) e^{|\Im \sqrt{z^2 - i \gamma} |x} \exp \br*{ \int_0^x (1 + t) |q(t)| \d t }.
    \end{equation*}
\end{lemma}
\begin{proof}
Let $\mu = \mu(z) := \sqrt{z^2 - i \gamma}$.
$\theta$ and $\theta'$ satisfy the integral equations
\begin{equation*}
    \theta(x,z) = \frac{\sin(\mu x)}{\mu} + \int_0^x \frac{\sin (\mu (x - t))}{\mu} q(t) \theta(t, z) \d t
\end{equation*}
and 
\begin{equation*}
    \theta'(x,z) = \cos(\mu x) + \int_0^x \cos(\mu (x - t)) q(t) \theta(t,z) \d t
\end{equation*}
hence satisfy the integral inequality
\begin{equation*}
    |\theta(x,z)| + |\theta'(x,z)| \leq (1 + x) e^{|\Im \mu| x} \sbr*{ 1 + \int_0^x e^{- |\Im \mu| t} |q(t)| \br*{|\theta(t,z)| + |\theta'(t,z)| } \d t }
\end{equation*}
where we used the fact that $|\sin(\mu x)||\mu|^{-1} \leq x e^{|\Im \mu|x}$ and $|\cos(\mu x)| \leq e^{|\Im \mu|x }$.
The result follows from an application of Gr\"onwall's Lemma.
\end{proof}

\subsection{Compactly Supported Potentials}

\begin{assumption}\label{ass:compact}
$q$ is compactly supported, that is, there exists $Q > 0$ such that
\begin{equation*}
    \supp\, q \subset [0,Q].
\end{equation*}

\end{assumption}

If Assumption \ref{ass:compact} holds and then the Jost solution $\varphi$ satisfies 
\begin{equation}\label{eq:jost-entire}
    \varphi(R,z) = e^{i z R}\qquad (R > Q, z \in \C)
\end{equation}
hence, for each $x \in [0,\infty)$, $\varphi(x,\cdot)$ can be analytically continued to $\C$.
Consequently, for $R > Q$, $f_R$ can be analytically continued to $\C$ and can be written as 
\begin{equation}\label{eq:fR-entire}
 f_R(z) =  z \theta(R,z) + i \theta'(R,z) \qquad (z \in \C).
\end{equation}

\begin{theorem}\label{thm:compact-number}
Suppose that Assumption \ref{ass:compact} holds. Then there exists $R_0 = R_0(q,\gamma) > 0$ such that for every $R \geq R_0$,
\begin{equation}\label{eq:number-compact-ineq}
    N(H_R) \leq \frac{11}{\log 2} \frac{ \gamma R^2}{\log R}.
\end{equation}
\end{theorem}

\begin{proof}

Let $z_0 \in \C_+$ be such that
\begin{equation}\label{eqpr:number-01}
    \psi_+(0,z_0^2 - i\gamma) \neq 0, \quad \Im \sqrt{z_0^2 - i \gamma} \geq 1 \quad \text{and} \quad  \sqrt{|z_0^2 - i \gamma|} \leq 2.
\end{equation}
In fact, by choosing $z_0$ to be the minimiser of some suitable total order on $\C$ in the set of points that maximise $z \mapsto \psi_+(0,z^2 - i\gamma)$ while satisfying the latter two inequalities of \eqref{eqpr:number-01}, $z_0$ can be determined uniquely by $q$ and $\gamma$, $z_0 = z_0(q,\gamma)$.
Define $r = r(R) > 0$ by
\begin{equation}\label{eqpr:number-02}
    \frac{r}{2} = \gamma^{1/2} + |z_0| + \frac{5 \gamma R}{\log R}.
\end{equation}
By the triangle inequality,
\begin{equation}
 |z - z_0| \leq |z| + |z_0| \leq \sqrt{|z^2 - i \gamma| + \gamma} + |z_0| \leq \sqrt{|z^2 - i \gamma|} + \gamma^{1/2} + |z_0|
\end{equation}
so,
\begin{equation}\label{eqpr:number-0}
S_R := \set*{z \in \C_+ : \sqrt{|z^2 - i \gamma|} \leq \frac {5 \gamma R}{\log R } } \subseteq \overline{B}_{r/2}(z_0). 
\end{equation}
Let $R > Q > 0$  be large enough so that estimate \eqref{eq:mag-bound-q-indpt} of Theorem \ref{thm:mag-bound} (b) holds.
Since the zeros of $f_R$ in $\C_+$ have a bijective correspondence with the eigenvalues of $H_R$, the set $S_R$ contains all the zeros of $f_R$ in $\C_+$ and hence the number of eigenvalues of $H_R$ is bounded by the number of zeros for $f_R$ in the ball $\overline{B}_{r/2}(z_0)$,
\begin{equation}\label{eqpr:number-1}
    N(H_R)  \leq |f_R^{-1}\set{0} \cap \overline{B}_{r/2}(z_0))|.
\end{equation}
Since $f_R$ is entire, Jensen's formula gives us 
\begin{equation}\label{eqpr:number-2}
    |f_R^{-1}\set{0} \cap \overline{B}_{r/2}(z_0)| \leq \frac{1}{\log 2} \log \abs*{ \frac{1}{f_R(z_0)} \sup_{|z - z_0| = r} \abs*{f_R(z)}}.
\end{equation}

Since $R > Q$, the terms $\mathcal E_1(R,z)$ and $\mathcal E_1(R,z)$, defined by $\eqref{eq:big-error-cont-E1}$ and $\eqref{eq:big-error-cont-E2}$ respectively, vanish.
Hence, by Lemma \ref{lem:cont-fR-decomp} and the fact that $\Im \sqrt{z_0^2 - i\gamma} \geq 1$, 
\begin{multline*}
    |f_R(z_0) u(z_0)| \geq |\psi_+(0,z_0^2 - i \gamma) \br{z_0 + \sqrt{z_0^2 - i \gamma} }| e^R  \\
     -  | \psi_-(0,z_0^2 - i \gamma) (z_0 - \sqrt{z_0^2 - i \gamma})) |e^{-R}.
\end{multline*}
Note that $\Im \sqrt{z_0^2 - i\gamma} \geq 1$ implies that $z_0 \neq \pm \sqrt{i \gamma}$ so Lemma \ref{lem:cont-fR-decomp} is indeed applicable here. 
Then, since $\psi_+(0,z_0^2 - i\gamma) \neq 0$,  $\sqrt{|z_0^2 - i \gamma|} \leq 2$ and $z_0 = z_0(q,\gamma)$,
\begin{equation}\label{eqpr:number-3}
|f_R(z_0)| \geq C(q,\gamma)
\end{equation}
for large enough $R$.

By expression \eqref{eq:fR-entire} for $f_R$ and the estimates in Lemma \ref{lem:gronwall-est} for $\theta$ and $\theta'$, for all $z \in \partial B_r(z_0)$,
\begin{equation}\label{eqpr:number-4}
    |f_R(z)| \leq C(q)(1 + R)(1 + |z|)e^{|\sqrt{z^2 - i \gamma}|R}.
\end{equation}
Furthermore, by the triangle inequality and expression \eqref{eqpr:number-02} for $r$, for all $z \in \partial B_r(z_0)$,
\begin{equation}\label{eqpr:number-5}
    \sqrt{|z^2 - i \gamma|} \leq |z - z_0| + |z_0| + \gamma^{1/2} = 3 \gamma^{1/2} + 3 |z_0| + \frac{10 \gamma R}{\log R}.
\end{equation}
Noting that for $z \in \partial B_r(z_0)$, the factor $(1 + |z|)$ in \eqref{eqpr:number-4} is $o(R)$, combining \eqref{eqpr:number-1} - \eqref{eqpr:number-5} gives us 
\begin{equation*}
    N(H_R) \leq \frac{1}{\log 2} \br*{ \log o(R^2) + \br*{3 \gamma^{1/2} + 3 |z_0|}R + \frac{10 \gamma R^2}{\log R}}
\end{equation*}
as $R \to \infty$.
Estimate \eqref{eq:number-compact-ineq} follows.
\end{proof}

\subsection{Exponentially Decaying Potentials}

\begin{assumption}[Naimark Condition]\label{ass:naimark}
There exists $a > 0$ such that 
\begin{equation*}
    \int_0^\infty e^{4 at} |q(t)| \d t < \infty.
\end{equation*} 
\end{assumption}

If Assumption \ref{ass:naimark} is satisfied then for each $x > 0$ the functions $\varphi(x,\cdot)$ and $\varphi'(x,\cdot)$ admit analytic continuations from $\C_+$ into $\set{\Im z > - 2 a}$. 
For each $x > 0$, the functions $E$ and $E^d$ appearing in the decomposition \eqref{eq:phi-decomp} of the Jost solution $\varphi$ satisfy
\begin{equation}\label{eq:E-continued-Linf-bound}
    |E(x,z)|+|E^d(x,z)| \leq C(q)\quad \text{if} \quad \Im z \geq - a
\end{equation}
and 
\begin{equation}\label{eq:E-continued-1z-ineq}
    |E(x,z)|+|E^d(x,z)| \leq \frac{C(q)}{|z|} \quad \text{if} \quad \Im z \geq - a \text{ and } |z| \geq 1.
\end{equation}
See \cite[Theorem 2.6.1]{Naimark} and \cite[Lemma 1]{stepin2015complex} for proofs of the above claims.

The next proposition allows us to utilise the uniform enclosure of Theorem \ref{thm:mag-bound} (a) in the estimation of the number of eigenvalues of $H_R$.
\begin{proposition}\label{prop:Jensen-Upper}
Suppose that $f$ is an analytic function defined on an open neighbourhood of the closed semi-disc $D_r := \overline{B}_r(0)\cap \overline{\C}_+$ for some $r > 0$.
Let $\alpha$ and $\beta$ be any numbers in the interval $(0,1)$ satisfying
\begin{equation}\label{eq:alpha-1-2-cond}
    \beta \br*{\frac{1 - \alpha}{\alpha + \beta}}^2 > \frac{Y}{\eta}
\end{equation}
and let $N(\alpha r)$ denote the number of zeros in the region 
\begin{equation}\label{eq:strip-semi-disc}
    D_{\alpha r,\eta,Y} := \set{z \in \C : \eta \leq \Im z \leq Y, \, |z| \leq \alpha r}
\end{equation}
where $Y,\eta > 0$ are given parameters satisfying $\eta < Y < r$. 
Then,
\begin{equation}\label{eq:Jensen-Upper-ineq}
N(\alpha r) \leq \frac{2}{\log \Lambda(r)} \log \br*{ \frac{1}{\min \set{\beta, 1 - \beta}} \frac{\sup_{z \in \partial D_r} |f(z)|}{|f(i \beta r)|}}
\end{equation}
where
\begin{equation}\label{eq:Lambda-defn}
    \Lambda(r) := \frac{1 + \frac{4 \beta \eta}{(\alpha + \beta)^2}\frac{1}{r} }{1 + \frac{4 Y}{(1 - \alpha)^2}\frac{1}{r}}.
\end{equation}

\end{proposition}

\begin{remark}\label{rem:Jensen}
One can always guarantee that condition \eqref{eq:alpha-1-2-cond} for $\alpha$ and $ \beta$ is satisfied by choosing, for instance, 
\begin{equation}\label{eq:alpha-choice}
    \alpha = \beta = \frac{1}{4}\frac{\eta}{2Y + \eta}.
\end{equation}
\end{remark}

\begin{figure}[h]
  \includegraphics[width=0.9\textwidth]{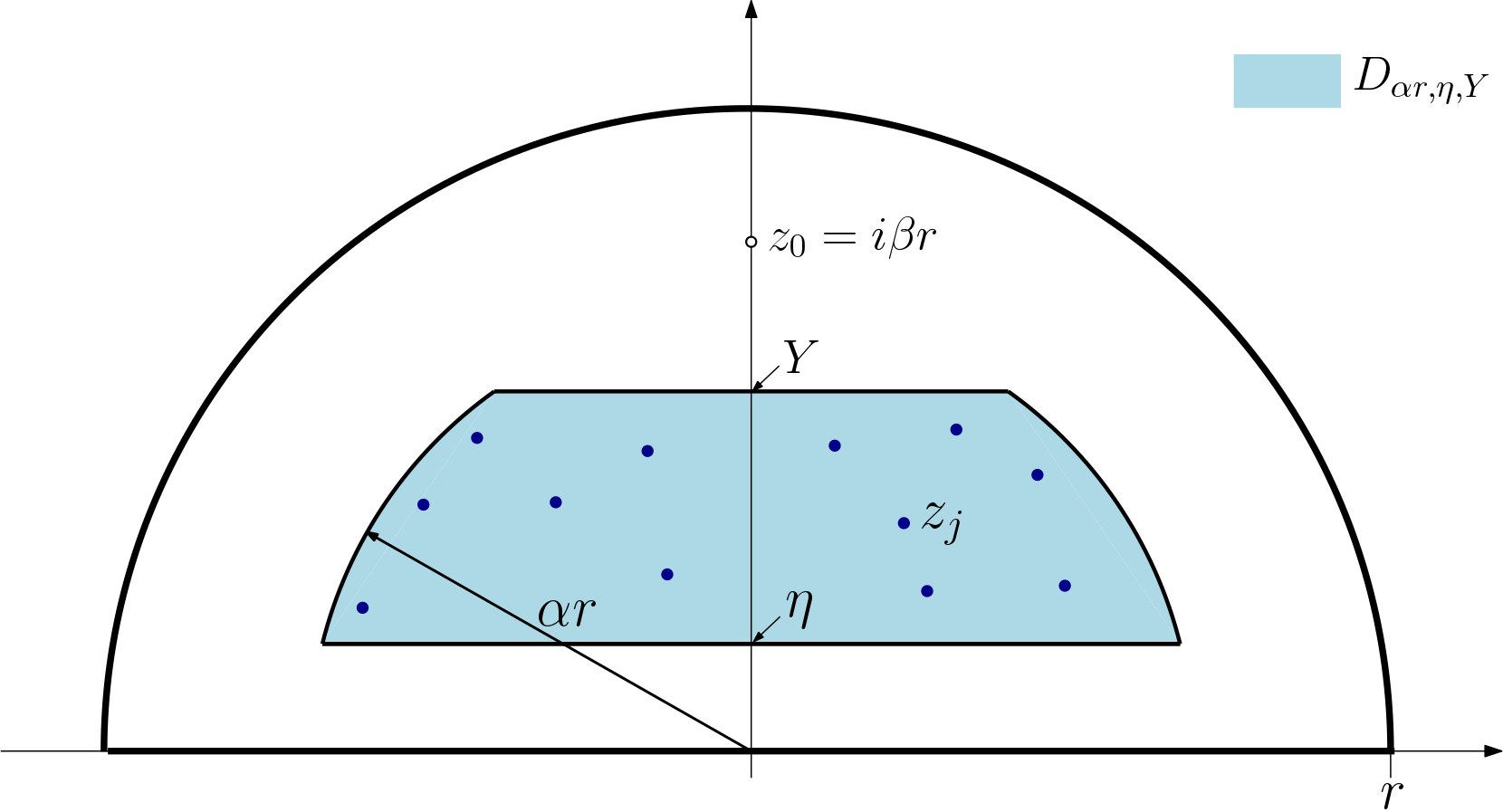}  
  \caption{Illustration for the setup of Proposition \ref{prop:Jensen-Upper}.}
  \label{fig:zeros}
\end{figure}

\begin{proof}[Proof of Proposition \ref{prop:Jensen-Upper}]
Let $\set{z_j}_{j=1}^{N(\alpha r)}$ denote the set of zeros of $f$ in the set $D_{\alpha r,\eta,Y}$ and consider the Blaschke product
\begin{equation*}
b(z) := \prod_j \frac{z - \overline{z}_j}{z - z_j} \equiv \prod_j b_j(z).    
\end{equation*}
Note that higher multiplicity zeros of $f$ are repeated in the set $\set{z_j}$ accordingly.
Let $z_0 := i \beta r$. 
The function $f(z)b(z)$ is analytic on an open neighbourhood of $D_r$ so by Cauchy's formula, 
\begin{equation}\label{eqpr:Jensen-0}
    \frac{1}{2 \pi i} \oint_{\partial D_r} \frac{f(z)b(z)}{z - z_0} \d z= f(z_0) b(z_0).
\end{equation}
Observing that $|z - z_0| \geq \min \set{\beta, 1 - \beta} r$ for all $z \in \partial D_r$, it holds that
\begin{equation*}
 \frac{1}{2 \pi} \oint_{\partial D_r} \frac{|\d z|}{|z - z_0|}  \leq \frac{1}{\min \set{\beta, 1 - \beta}}
\end{equation*}
which can be used to estimate the integral in \eqref{eqpr:Jensen-0} to get 
\begin{equation}\label{eqpr:Jensen-1}
    \prod_j \frac{|b_j(z_0)|}{\sup_{z \in \partial D_r } |b_j(z)|} \leq \frac{\sup_{z \in \partial D_r} |f(z)|}{|f(z_0)|} \frac{1}{\min \set{\beta, 1 - \beta}}.
\end{equation}
By a direct computation, we have
\begin{equation*}
    |b_j(z)| = \sqrt{1 + \frac{4 \Im z \Im z_j}{|z - z_j|^2}}.
\end{equation*}
Since
\begin{equation*}
    \Im z_0 = \beta r,\quad \Im z_j \geq \eta, \quad  |z_0 - z_j| \leq (\alpha + \beta)r,
\end{equation*}
giving us a lower bound for $|b_j(z_0)|$, and since for any $z \in \C$ with $|z| = r$
\begin{equation*}
    \Im z \leq r, \quad \Im z_j \leq Y, \quad |z - z_j| \geq (1 - \alpha)r,
\end{equation*}
giving us an upper bound for $|b_j(z)|$, we have 
\begin{equation}\label{eqpr:Jensen-2}
    \frac{|b_j(z_0)|}{|b_j(z)|} \geq \Lambda(r)^{1/2}
\end{equation}
for any $ z \in \partial D_r$ with $|z| = r$.
Furthermore, if $z \in \R$ then $|b_j(z)| = 1$ so \eqref{eqpr:Jensen-2} in fact holds for every $ z \in \partial D_r$.
Combining \eqref{eqpr:Jensen-2} with \eqref{eqpr:Jensen-1} gives us 
\begin{equation}\label{eqpr:Jensen-3}
    \Lambda(r)^{N(\alpha r)/2} \leq \frac{1}{\min \set{\beta, 1 - \beta}} \frac{\sup_{z \in \partial D_r} |f(z)|}{|f(z_0)|}.
\end{equation}
If hypothesis \eqref{eq:alpha-1-2-cond} for $\alpha$ and $\beta$ holds then $\Lambda(r) > 1$ so we can take the logarithm of both sides of \eqref{eqpr:Jensen-3} and rearrange to obtain inequality \eqref{eq:Jensen-Upper-ineq}.
\end{proof}

\begin{theorem}\label{thm:exp-number}

Suppose that Assumption \ref{ass:naimark} holds. Then there exists $R_0 = R_0(q,\gamma) > 0$ such that for every $R \geq R_0$,
\begin{equation}\label{eq:exp-number-ineq}
    N(H_R) \leq C \frac{\sqrt{X} + a}{a^2}\frac{\gamma^2 R^3}{(\log R)^2}
\end{equation}
where $C = 88788$ and $X = X(q,\gamma) > 0$ is the constant appearing in Theorem \ref{thm:mag-bound} (a).
\end{theorem}

\begin{proof}

Let $\tilde{f}_R(z) := f_R(z - i a)$ and let $\alpha, \beta > 0$ satisfy equation \eqref{eq:alpha-choice} of Remark \ref{rem:Jensen} with $\eta = a$ and $Y = \sqrt{X} + a$ where $X = X(q,\gamma)$ is the constant appearing in Theorem \ref{thm:mag-bound}.
Then hypothesis \eqref{eq:alpha-1-2-cond} of Proposition \ref{prop:Jensen-Upper} is satisfied.
Note that with this choice of $\beta$ we have $\beta < 1/2$, so,
\begin{equation}
    \min \set{\beta, 1 - \beta} = \beta.
\end{equation}

The zeros of $\tilde{f}_R$ in $\set{\Im z > a}$ have a bijective correspondence to eigenvalues of $H_R$ given by 
\begin{equation}
    (z - i a)^2 \in \sigdis(H_R) \qquad \iff \qquad \Im z > a \text{ and }\tilde{f}_R(z) = 0 .
\end{equation}
Assuming without loss of generality that $X \geq \gamma$, the square root of the enclosure $B_X(0) \cup \Gamma_\gamma$ is contained in the strip $\set{0 \leq \Im w \leq \sqrt{X}} \subset \C$. 
Then by the uniform enclosure of Theorem \ref{thm:mag-bound} (a), the zeros of $\tilde{f}_R$ in $\set{\Im z > a}$ are contained in the strip $\set{a \leq \Im z \leq \sqrt{X} + a}$.
By the triangle inequality and the magnitude bound of Theorem \ref{thm:mag-bound} (b), any zero $z$ of $\tilde{f}_R$ with $\Im z > a$ satisfies
\begin{equation}
    |z| \leq  \gamma^{1/2} + a + \sqrt{| (z - ia)^2 - i \gamma|} \leq \alpha r 
\end{equation}
where $r = r(R)$ is defined by
\begin{equation}\label{eqpr:exp-number-0}
    \alpha r = \gamma^{1/2} + a + \frac{5 \gamma R}{\log R }.
\end{equation}
Hence the zeros of $\tilde{f}_R$ in $\set{\Im z > a}$ are contained in $D_{\alpha r, \eta, Y}$.

Applying Proposition \ref{prop:Jensen-Upper}, we get an estimate for the number of eigenvalues of $H_R$,
\begin{equation}\label{eqpr:exp-number-1}
    N(H_R) = |\tilde{f}_R^{-1}\set{0}\cap D_{\alpha r,\eta,Y}| \leq \frac{2}{\log \Lambda(r)} \log \br*{\frac{1}{\beta} \frac{\sup_{z \in \partial D_r} |\tilde{f}_R(z)|}{|\tilde{f}_R(i\beta r)|}}
\end{equation}
where 
\begin{equation}\label{eqpr:exp-number-2}
    \Lambda(r) = \frac{1 + C_1/r}{1+ C_2/r}
\end{equation}
for some constants $C_1 > C_2 > 0$ depending only on $X$ and $a$.
The remainder of the proof consists in estimating the right hand side of \eqref{eqpr:exp-number-1}.

Let $z_R :=  i \beta r(R) - i a$.
By Lemma \ref{lem:cont-fR-decomp},
\begin{multline}\label{eqpr:exp-number-3}
    |f_R(z_R) u(z_R)| \geq  |\psi_+(0,z_R^2 - i \gamma)(z_R + \sqrt{z_R^2 - i \gamma} + \mathcal E_2(R,z_R))|  \\
     - |\psi_-(0,z_R^2 - i \gamma)(z_R - \sqrt{z_R^2 - i \gamma} + \mathcal E_1(R,z_R))|  
\end{multline}
for large enough $R$.
By estimates \eqref{eq:Epm-error} for $E_\pm$ and $E^d_\pm$, and the corresponding estimates \eqref{eq:E-continued-1z-ineq} for $E$ and $E^d$,
\begin{equation}\label{eqpr:exp-number-4a}
    |u(z_R)|+|\psi_-(0,z_R^2 - i \gamma)|+|\mathcal E_1(R,z_R)|+|\mathcal E_2(R,z_R)| \leq C(q,\gamma) 
\end{equation}
and 
\begin{equation}\label{eqpr:exp-number-4b}
    |\psi_+(0,z_R^2 - i \gamma)| \geq C(q,\gamma) 
\end{equation}
for large enough $R$. 
By Lemma \ref{lem:h-pm},
\begin{equation}\label{eqpr:exp-number-5}
    \lim_{R \to \infty} |z_R + \sqrt{z_R^2 - i \gamma}| = \infty \quad \text{and}\quad \lim_{R \to \infty} |z_R - \sqrt{z_R^2 - i \gamma}| = 0.
\end{equation}
Combining \eqref{eqpr:exp-number-3}  with \eqref{eqpr:exp-number-4a}, \eqref{eqpr:exp-number-4b} and \eqref{eqpr:exp-number-5} gives us 
\begin{equation}\label{eqpr:exp-number-6}
    |\tilde{f}_R(i \beta r)| = |f_R(z_R)| \geq 1 
\end{equation}
for large enough $R$. 

The factor involving $\Lambda(r)$ on the right hand side of \eqref{eqpr:exp-number-1} can be estimated using the expression \eqref{eqpr:exp-number-2} for $\Lambda$ and the inequality $\log x \geq (x - 1)/(x + 1)$ $(x \geq 1)$,
\begin{equation}\label{eqpr:exp-number-7}
    \log \Lambda(r) \geq \frac{\Lambda(r) - 1}{\Lambda(r) + 1} = \frac{(C_1 - C_2)/r(R)}{2 + (C_1 + C_2)/r(R)} \geq \frac{C_1 - C_2}{3 r(R)}
\end{equation}
for large enough $R$.

The function $\tilde{f}_R$ is estimated from above using the bound in Lemma \ref{lem:gronwall-est} for $\theta$ and $\theta'$ and the uniform bounds $\eqref{eq:E-continued-Linf-bound}$ for $E(R,\cdot)$ and $E^d(R,\cdot)$,
\begin{equation}\label{eqpr:exp-number-8}
    |\tilde{f}_R(z)| \leq C(q)\br*{1 + R} \br*{ 1 + |z| }e^{a R} e^{|\sqrt{(z -  i a)^2 - i \gamma}|R} \qquad (z \in \overline{\C}_+). 
\end{equation}
Using the expression \eqref{eqpr:exp-number-0} for $r$,  for any $z \in \partial D_r$ we have 
\begin{equation}\label{eqpr:exp-number-9}
    \sqrt{|(z - i a)^2 - i \gamma|} \leq \gamma^{1/2} + a + |z| \leq O(1) + \frac{5 \gamma R}{\alpha \log R }
\end{equation}
as $R \to \infty$.
Combining \eqref{eqpr:exp-number-1} with  \eqref{eqpr:exp-number-6}, \eqref{eqpr:exp-number-7}, \eqref{eqpr:exp-number-8} and \eqref{eqpr:exp-number-9}, noting that $|z| = o(R)$ for $z \in \partial D_R$ and $\beta^{-1} =O(1)$, gives
\begin{equation*}
    N(H_R) \leq \frac{6}{C_1 - C_2} \br*{ O(1) + \frac{5 \gamma R}{\alpha \log R }}\br*{ O(R) +  \frac{5 \gamma R^2}{ \alpha \log R} }
\end{equation*}
as $R \to \infty$ and so 
\begin{equation}\label{eqpr:exp-number-10}
    N(H_R) \leq \frac{151 \gamma^2 R^3}{(C_1 - C_2)\alpha^2 (\log R)^2}
\end{equation}
for large enough $R$.

Finally, we put the constant into a more illuminating form. 
By the definition \eqref{eq:Lambda-defn} of $\Lambda$ in Proposition \ref{prop:Jensen-Upper},
\begin{equation}\label{eqpr:exp-number-11}
    C_1 = \frac{\eta}{\alpha} \quad \text{and} \quad C_2 = \frac{4 Y}{(1 - \alpha)^2}.
\end{equation}
Since $\frac{\eta}{12 Y} \leq \alpha \leq \frac{\eta}{8 Y}$, we have 
\begin{equation}\label{eqpr:exp-number-12}
    (C_1 - C_2)\alpha^2 = \eta \alpha - \frac{4Y \alpha^2}{(1 - \alpha)^2} \geq \frac{\eta^2}{12Y} - \frac{4 Y \alpha^2}{(1 - \frac{\eta}{8Y})^2} 
\end{equation}
and since $0 \leq \eta/Y \leq 1$, we have
\begin{equation}\label{eqpr:exp-number-13}
    \frac{\alpha^2}{(1 - \frac{\eta}{8Y})^2} = \frac{\eta^2}{64Y^2}\frac{1}{(1 + \frac{\eta}{2Y})^2 (1 - \frac{\eta}{8Y})^2} \leq \frac{\eta^2}{49 Y^2}
\end{equation}
Combining \eqref{eqpr:exp-number-12} and \eqref{eqpr:exp-number-13}, we have
\begin{equation}
    (C_1 - C_2)\alpha^2 \geq \frac{1}{588}\frac{\eta^2}{Y}.
\end{equation}
which gives estimate \eqref{eq:exp-number-ineq} when substituted into \eqref{eqpr:exp-number-10}, with $Y = \sqrt{X} + a$ and $\eta = a$.

\end{proof}

\bibliographystyle{abbrv}
\bibliography{DBMagnRefs.bib}

\end{document}